\documentclass[final,reqno,10pt]{amsart}

\usepackage{amsmath,amssymb,amsthm,bbm}

\usepackage[latin1]{inputenc} % `traduz' os acentos do portug

\usepackage{a4wide}

\usepackage{enumitem,epsfig,psfrag}
\setenumerate{leftmargin=*}

\usepackage{color}

\usepackage[colorlinks,bookmarks,linkcolor=black,citecolor=black]{hyperref}
\usepackage[notref,notcite,color]{showkeys}

\let\subset\subseteq
\let\eps\varepsilon
\let\rho\varrho
\def\dcup{\dot\cup}

\def\Prob{\mathbb{P}}

\def\cX{\mathcal{X}}
\def\cY{\mathcal{Y}}
\def\cZ{\mathcal{Z}}
\def\cJ{\mathcal{J}}
\def\cR{\mathcal{R}}

\def\itmit#1{\rm ({\it #1\,})}
\def\rom{\itmit{\roman{*}}}
\def\abc{\itmit{\alph{*}}}

\newtheorem{theorem}{Theorem}
\newtheorem{lemma}[theorem] {Lemma}
\newtheorem{corollary}[theorem] {Corollary}
\newtheorem{conjecture}[theorem] {Conjecture}

\newtheorem{definition}[theorem] {Definition}
\theoremstyle{remark}

\newtheorem*{remark*}{Remark}
\newtheorem{AuxClaim}[theorem]{Claim}

\definecolor{Red}{cmyk}{0,1,1,0}
\definecolor{Magenta}{cmyk}{0,1,0,0}
\definecolor{Blue}{cmyk}{1,1,0,0}

\newcommand{\oldequation}{ }
\newenvironment{reequation}[1]{
  \renewcommand{\oldequation}{\theequation}
  \renewcommand{\theequation}{#1'}
  \begin{equation}
  }{
  \end{equation}
  \renewcommand{\theequation}{\oldequation} 
  \addtocounter{equation}{-1}
  \par\noindent%
  \ignorespacesafterend
}

\newcommand{\By}[2]{\overset{\mbox{\tiny{#1}}}{#2}}
\newcommand{\ByRef}[2]{   \By{\eqref{#1}}{#2} }

\newcommand{\eqByRef}[1]{ \ByRef{#1}{=} }

\newcommand{\leByRef}[1]{ \ByRef{#1}{\le} }

\newcommand{\BAR}[1]{{\smash{\overset{\scalebox{2}[1.2]{\!\raisebox{-.5pt}{\_}}}{\smash{#1}\vphantom{X}}}}}

\newcommand{\neighbour}{\Gamma}
\newcommand{\cneighbour}{\Gamma^*}
\newcommand{\Inj}{\mathcal{I}}

\DeclareMathOperator{\dis}{dis}
\DeclareMathOperator{\its}{int}

\title[Coloured copies of large graphs with small maximum degree]
  {Properly coloured copies and rainbow copies of large graphs with small
   maximum degree}

\author{Julia B\"ottcher}
\address{Instituto de Matem\'atica e Estat\'{\i}stica, Universidade de
 S\~ao Paulo, Rua do Mat\~ao 1010, 05508--090~S\~ao Paulo, Brazil}
\email{julia@ime.usp.br}

\author{Yoshiharu Kohayakawa}
\address{Instituto de Matem\'atica e Estat\'{\i}stica, Universidade de
 S\~ao Paulo, Rua do Mat\~ao 1010, 05508--090~S\~ao Paulo, Brazil}
\email{yoshi@ime.usp.br}

\author{Aldo Procacci}
\address{
  Departamento de Matem\'atica,
  Universidade Federal de Minas Gerais,
  Avenida Ant\^onio Carlos 6627, Caixa Postal 702, 30161--970~Belo Horizonte,
  Brazil
}
\email{aldo@mat.ufmg.br}

\thanks{ 
  The first author was partially supported by FAPESP (Proc.~2009/17831-7).
  The second author was partially supported by CNPq (Proc.~308509/2007-2).
  The third author was partially supported by CNPq (Proc.~3302517/2007-3).
  The authors are grateful to NUMEC/USP, Núcleo de Modelagem Estocástica e
  Complexidade of the University of São Paulo, for supporting this
  research.
}

\date{\today}

\begin{document}

\begin{abstract}
  Let~$G$ be a graph on~$n$ vertices with maximum degree~$\Delta$.  We use
  the Lov\'asz local lemma to show the following two results about
  colourings~$\chi$ of the edges of the complete graph~$K_n$.
  If for each vertex~$v$ of~$K_n$ the colouring~$\chi$ assigns each colour
  to at most $(n-2)/22.4\Delta^2$ edges emanating from~$v$, then there is a
  copy of~$G$ in~$K_n$ which is properly edge-coloured by~$\chi$.  This
  improves on a result of Alon, Jiang, Miller, and Pritikin [Random
  Struct.\ Algorithms~23(4), 409--433, 2003].
  On the other hand, if~$\chi$ assigns each colour to at most
  $n/51\Delta^2$ edges of~$K_n$, then there is a copy of~$G$ in~$K_n$ such
  that each edge of~$G$ receives a different colour from~$\chi$.  This
  proves a conjecture of Frieze and Krivelevich [Electron.\ J.~Comb.~15(1),
  R59, 2008].

  Our proofs rely on a framework developed by Lu and Sz\'ekely
  [Electron.\ J.~Comb.~14(1), R63, 2007] for applying the local lemma
  to random injections. 
  In order to improve the constants in our results we use a version of
  the local lemma due to Bissacot, Fern\'andez, Procacci, and Scoppola
  [preprint, arXiv:0910.1824].

\end{abstract}

\maketitle

\section{Results}

The canonical Ramsey theorem~\cite{ErdRad}  states that for any graph~$G$, each
colouring of the edges of the complete graph~$K_n$ contains at least one of
the following types of $G$-copies: a monochromatic copy of~$G$; a copy
of~$G$ such that each edge has a different colour; a copy of~$G$ such
that, after we order the vertices of~$G$ appropriately, 
the colour of each edge is determined solely by the first vertex of this
edge with respect to this ordering.
The classical Ramsey theorem~\cite{Ramsey} is a special case of this result
where the \emph{number} of colours used in the colouring is restricted.  In this
paper we consider restrictions on the edge colourings of~$K_n$ which are of
different nature and show that under such restrictions certain copies of
\emph{spanning} graphs~$G$ exist.

Let $F$ be a graph. 
% In the following we consider colourings of the edges
% of~$F$, which we occasionally simply call colouring.  
An edge
colouring~$\chi$ of~$F$ is called \emph{$k$-bounded} if it does not use
any colour more than~$k$ times. It is called \emph{locally $k$-bounded} if
for every vertex~$v$ of~$F$ it does not use any colour more than~$k$ times
on the edges incident to~$v$. We say that~$\chi$ is \emph{proper} if
intersecting edges of~$F$ receive different colours, and that it is
\emph{rainbow} if all edges of~$F$ receive different colours.

Given a graph~$G$ on~$n$ vertices we will consider problems of the
following type: For which numbers~$k$ does the complete graph~$K_n$ have
one of the following properties?
\begin{enumerate}[label=\abc]
  \item\label{qu:proper}
    Every locally $k$-bounded colouring~$\chi$ of~$K_n$ contains a copy
   of~$G$ which is properly coloured by~$\chi$.
 \item\label{qu:rainbow}
    Every $k$-bounded colouring~$\chi$ of~$K_n$ contains a copy
    of~$G$ which is rainbow coloured by~$\chi$.
\end{enumerate}
We also say that a colouring of~$K_n$ is \emph{$G$-proper} if it has the property asserted
by~\ref{qu:proper}, and \emph{$G$-rainbow} if it has the property asserted
by~\ref{qu:rainbow}.

\medskip

Our proofs apply a probabilistic method and rely on the Lov\'asz local
lemma~\cite{ErdLov75}. 

\subsection{Properly coloured subgraphs}

An old conjecture of Bollob\'as and Erd\H{o}s~\cite{BolErd76} states that
for $k=\lfloor n/2\rfloor$
every locally $k$-bounded colouring of~$K_n$ is $C_n$-proper,
where~$C_n$ is the cycle on~$n$ vertices (that is, a Hamilton cycle).
Bollob\'as and Erd\H{o}s showed that this is true for $k<n/69$. This was improved by
Shearer~\cite{Shearer79} to $k<n/7$, and by Chen and Daykin~\cite{CheDay} to
$k\le n/17$. Currently, the best result is due to Alon and
Gutin~\cite{AloGut}, who showed that for any~$\eps>0$ there is an~$n_0$ such
that for all $n\ge n_0$ we can choose $k=\lfloor(1-2^{-1/2}-\eps)n\rfloor$.

Generalising from~$C_n$ to a larger class of graphs, Shearer~\cite{Shearer79}
proposed the following conjecture. A \emph{cherry} in a graph~$G$ is a
copy of a path of length~$2$ in~$G$. 

\begin{conjecture}[Shearer~\cite{Shearer79}]
\label{con:Shearer}
  For fixed~$k$ and~$p$, if~$n$ is sufficiently large and if~$G$ is an
  $n$-vertex graph containing at most~$pn$ cherries, then
  every locally $k$-bounded edge colouring of~$K_n$ is $G$-proper.
\end{conjecture}

This conjecture considers the case of constant~$k$, but allows graphs~$G$ with
maximum degree~$\Omega(\sqrt{n})$. 
Replacing the global condition on the number of cherries in~$G$ by a local condition,
Alon et al.~\cite{AloJiaMilPri} proved that for fixed~$k$ and~$\Delta$,
if~$n$ is sufficiently large and~$G$ is an $n$-vertex graph with maximum
degree~$\Delta$ then every locally $k$-bounded edge colouring of~$K_n$ is
$G$-proper. More precisely they showed the following result.

\begin{theorem}[Alon, Jiang, Miller, Pritikin~\cite{AloJiaMilPri}]
  Let~$G$ be an $n$-vertex graph with maximum degree~$\Delta$. If~$k$
  satisfies $216(3k+2\Delta)^7(\Delta+1)^{20}k<n$ then any locally
  $k$-bounded edge colouring of~$K_n$ is $G$-proper.
\end{theorem}

Notice that for graphs~$G$ whose maximum degree is bounded by a constant,
the parameter~$k$ can be of order $\Omega(n^{1/8})$ in this theorem.  Alon
et al.\ also suspected that it should be possible to replace this bound
on~$k$ by a linear bound.

As a consequence of the following theorem we show in
Corollary~\ref{cor:proper} that this is indeed possible. The theorem which
implies this corollary uses the global condition on the total number of
cherries from the conjecture of Shearer together with a local condition on
the distribution of these cherries.

\begin{theorem}
\label{thm:proper}
  Let~$G$ be a graph on~$n$ vertices containing at most~$pn$ cherries such
  that each vertex is contained in at most~$q$ cherries. For $k\le
  \frac13\big(\frac56\big)^5(n-2)/(q+3p)$ every locally $k$-bounded edge
  colouring of $K_n$ is $G$-proper.
\end{theorem}

Observe that a graph with maximum degree~$\Delta$ contains at most
$\binom{\Delta}{2}n\le\frac12\Delta^2n$ cherries, and each of its vertices
is contained in at most $\binom\Delta2+\Delta(\Delta-1)\le\frac32\Delta^2$
cherries. Therefore we immediately obtain the following corollary.

\begin{corollary}
\label{cor:proper}
  If~$G$ has $n$ vertices and maximum degree~$\Delta>0$ then any locally
 $\big((n-2)/22.4\Delta^2\big)$-bounded edge colouring of~$K_n$ is $G$-proper.
\end{corollary}

As mentioned earlier, this corollary asserts that for bounded degree
graphs~$G$ we can even choose~$k$ linear in~$n$, which clearly is best possible up
to the constant. Moreover, the corollary implies that for each constant~$k$
there is a constant $c=c(k)>0$ such that we can even take~$G$ with maximum
degree $c\sqrt{n}$ in Theorem~\ref{thm:proper}. This means that
Conjecture~\ref{con:Shearer} holds for graphs~$G$ with maximum
degree~$c\sqrt{n}$. 

In comparison, all graphs with maximum
degree~$\Delta$ which contain at most~$pn$ cherries satisfy
$\sqrt{2pn}\ge\big(2\binom{\Delta}{2}\big)^{1/2}\approx\Delta$.
It follows that the graphs~$G$ covered by Conjecture~\ref{con:Shearer} have maximum
degree less than roughly $\sqrt{2pn}$, where~$p$ is a constant. So, in
terms of the maximum degree, our result achieves the `right' order of
magnitude.

\subsection{Rainbow subgraphs}

Erd\H{o}s and Stein posed the question of determining the largest~$k$ such
that every $k$-bounded edge colouring of~$K_n$ is $C_n$-rainbow
(see~\cite{ErdNesRod}). Hahn and Thomassen~\cite{HahTho} conjectured
that~$k$ can be linear. This was shown by Albert, Frieze, and
Reed~\cite{AlbFriRee95} (see also~\cite{Rue}) who improved on earlier
sublinear bounds by Erd\H{o}s, Ne\v{s}et\v{r}il, and
R\"odl~\cite{ErdNesRod}, by Hahn and Thomassen~\cite{HahTho}, and by Frieze
and Reed~\cite{FriRee}.

\begin{theorem}[Albert, Frieze, Reed~\cite{AlbFriRee95}]
\label{thm:Albert}
  If $k\le n/64$ then every $k$-bounded edge colouring of~$K_n$ is
  $C_n$-rainbow.
\end{theorem}

Frieze and Krivelevich~\cite{FriKri08} extended this result and showed that
there is a constant~$c>0$ such that for~$k\le cn$ every $k$-bounded edge
colouring of~$K_n$ is $C_\ell$-rainbow for all $3\le\ell\le n$. Moreover,
they considered almost spanning trees with bounded maximum degree and
proved the following theorem.

\begin{theorem}[Frieze, Krivelevich~\cite{FriKri08}]
  For every number $\eps>0$ and every integer $\Delta>0$ there is a
  constant $c>0$ such that the following holds for every tree~$T$ on at most
  $(1-\eps)n$ vertices with maximum degree~$\Delta$. If $k\le cn$ then
  every $k$-bounded edge colouring of~$K_n$ is $T$-rainbow.
\end{theorem}

Frieze and Krivelevich also showed that for a very special class of
\emph{spanning} trees~$T$ linearly bounded colourings of~$K_n$ are
$T$-rainbow. They conjectured that this is true for all bounded degree
trees.

In this paper we prove this conjecture. In fact, our result is more general
and shows that the conclusion of the conjecture holds for all graphs with
bounded maximum degree. This also improves on a result in~\cite{BoeKohTar}
which considers bipartite $n$-vertex graphs~$G$ with sublinear bandwidth
and constant maximum degree~$\Delta$, for~$n$ sufficiently large. The
result in~\cite{BoeKohTar} states that for $k\le c(n/\log n)^{1/\Delta}$,
where $c=c(\Delta)>0$, every $k$-bounded colouring of~$K_n$ is $G$-rainbow.

\begin{theorem}
\label{thm:rainbow}
  Let~$G$ be a graph on~$n$ vertices with maximum degree~$\Delta>0$. For $k\le
  n/(51\Delta^2)$ every $k$-bounded edge colouring of~$K_n$ is $G$-rainbow.
\end{theorem}

Again, Theorem~\ref{thm:rainbow} can be applied to graphs with growing
maximum degree; it asserts that constantly bounded edge
colourings force rainbow copies of all graphs with maximum degree
$c\sqrt{n}$ for some constant $c=c(k)>0$.

\section{The local lemma  and random injections}

Probabilistic existence proofs often try to 
estimate from above the probability that
any of a set of bad events $\{X_i\}_{i\in[t]}$ occurs, with the goal of showing that
$\Prob(\bigcup_{i\in[t]} X_i)<1$. In
many classical applications the \emph{union bound}
$\Prob(\bigcup_{i\in[t]} X_i)\le\sum_{i\in[t]}\Prob(X_i)$ is used for this
purpose, which is obviously only a good estimate when the bad events are
disjoint or almost disjoint. If the bad events, on the other hand, are
mutually independent then
$\Prob(\bigcap_{i\in[t]}\BAR{X}_i)=\prod_{i\in[t]}(1-\Prob(X_i))$, which
clearly implies that $\Prob(\bigcup_{i\in[t]} X_i)<1$ iff no bad
event occurs with probability~$1$.
The local lemma is a compromise between these two extremes: it takes dependencies
into account but gives a more optimistic upper bound for
$\Prob(\bigcup_{i\in[t]} X_i)$ than the union bound, provided the
dependencies are not too dense.

Before we can formulate this powerful tool we need some definitions.
For a graph $G=(V,E)$ and a vertex $v\in V$ we denote the
\emph{neighbourhood} of~$v$ by $\neighbour_G(v):=\{u\in V\colon uv\in E\}$ and
the \emph{closed neighbourhood}  by
$\cneighbour_G(v):=\neighbour_G(v)\cup\{v\}$. We may also omit the
subscript~$G$. We let $\Delta(G)$ denote the maximum degree of~$G$.

\begin{definition}[dependency graph]
  Let~$\cX$ be a set of events in some probability space. Then~$D=(\cX,E)$
  is a \emph{dependency graph} for~$\cX$ if every event~$X\in\cX$ is
  mutually independent of its non-adjacent events in~$D$, i.e., for every
  set of events~$\cZ\subset\cX \setminus \cneighbour_D(X)$ we have that
  $\Prob(X|\bigcap_{Y\in\cZ}\BAR{Z})=\Prob(X)$. We say that~$D$ is a
  \emph{negative dependency graph} if
  $\Prob(X|\bigcap_{Y\in\cY}\BAR{Z})\le\Prob(X)$ for all such~$X$
  and~$\cZ$.
\end{definition}

The local lemma states that, if a given set~$\cX$ of bad events has a
sufficiently sparse (negative) dependency graph compared to their
probabilities, then with positive probability no bad event occurs.

\begin{lemma}[Lov\'asz local lemma~\cite{ErdLov75,ErdSpe91}]
\label{lem:LLLL}
  Let $\cX=\{X_i\}_{i\in[t]}$ be a set of events with negative dependency
  graph~$D=(\cX,E)$. If
  \begin{enumerate}[label=\rom]
    \item\label{lem:LLLL:maxdeg} we have
      \[\Prob(X_i)<1/e(\Delta(D)+1) \qquad\text{for all $i\in[t]$}\,,\]
    \item\label{lem:LLLL:xi} or if there are numbers $\{x_i\}_{i\in[t]}$ in
      $(0,1)$ such that
      \begin{equation}\label{eq:LLLL}
        \Prob(X_i)\le x_i\prod\nolimits_{X_iX_j\in E}(1-x_j)
       \qquad\text{for all $i\in[t]$}\,,
      \end{equation}
  \end{enumerate}
  then $\Prob(\bigcap_{i\in[t]}\BAR{X}_i)>0$.
\end{lemma}

% see also \cite{AlonSpencer}
Originally, the local lemma was proved in~\cite{ErdLov75} with the stronger
assumption that~$D$ is a dependency graph. Erd\H{o}s and Spencer then
observed in~\cite{ErdSpe91} that essentially the same proof applies when
negative dependency graphs are used instead.

A quick calculation shows that Lemma~\ref{lem:LLLL}\ref{lem:LLLL:xi}
implies Lemma~\ref{lem:LLLL}\ref{lem:LLLL:maxdeg} (it suffices to take
$x_i:=1/(\Delta+1)$ for all~$i$).  Moreover, if in Lemma~\ref{lem:LLLL} we change
from the variables $\{x_i\}_{i\in[t]}$ to the variables
$\{\mu_i\}_{i\in[t]}$ defined by $x_i=\mu_i/(1+\mu_i)$, then
Condition~\eqref{eq:LLLL} becomes
\begin{equation}
\label{eq:LLLL:varchange}
  \Prob(X_i)
  \le \frac{\mu_i}{\prod\nolimits_{X_j\in\cneighbour_D(X_i)}(1+\mu_j)}
  =\frac{\mu_i}{\sum\nolimits_{R\subset\cneighbour_D(X_i)}
      \prod\nolimits_{X_j\in R}\mu_j} \;.
\end{equation}
As usual, an empty product has value~$1$.

Using connections between the local lemma, independent set polynomials, and
lattice gas partition functions, in~\cite{BisFerProSco} a version of the local lemma was recently
established which often improves the constants in the upper
bounds on~$\Prob(X_i)$. This lemma states that the
sum over all $R\subset\cneighbour_D(X_i)$ on the right hand side of~\eqref{eq:LLLL:varchange} can be
replaced by a sum over all such~$R$ which form an independent set.

\begin{lemma}[Bissacot, Fern\'andez, Procacci, Scoppola~\cite{BisFerProSco}]
\label{lem:LLLL:new}
  Let $\cX=\{X_i\}_{i\in[t]}$ be a set of events with (negative) dependency
  graph~$D=(\cX,E)$. For each $i\in[t]$ let~$\cR_i$ be the family of all
  subsets of $\cneighbour(X_i)$ which are independent sets. If
  \begin{enumerate}[label=\rom]
    \item\label{lem:LLLL:new:maxdeg} there is a positive number~$\mu$ such
      that
      \begin{equation}\label{eq:LLLL:new:maxdeg}
        \Prob(X_i)\le
        \frac{ \mu }{ 
          \sum_{R\in \cR_i}  \mu^{|R|}
        } 
        \qquad\text{for all $i\in[t]$}\,,
     \end{equation}
     % where in the sum the empty set yields a contribution of~$1$, 
    \item\label{lem:LLLL:new:mu} or if there are positive numbers $\{\mu_i\}_{i\in[t]}$ such that
      \begin{equation}\label{eq:LLLL:new:mu}
        \Prob(X_i)\le 
        \frac{\mu_i}{
          \sum_{R\in \cR_i} \prod_{X_j\in R}\mu_j 
        }
        \qquad\text{for all $i\in[t]$}\,,
      \end{equation}
     % where in the sum the empty set yields a contribution of~$1$, 
  \end{enumerate}
  then $\Prob(\bigcap_{i\in[t]}\BAR{X}_i)>0$.
\end{lemma}

A \emph{clique} in a graph~$G$ is the vertex set of a complete subgraph
of~$G$.  Note that Lemma~\ref{lem:LLLL:new} provides better bounds than
Lemma~\ref{lem:LLLL} when the neighbourhoods of events in the independence
graph are `dense', for instance, when they can be written as the union of
few cliques. Hence, in an an application of this lemma we shall aim at
decomposing the closed neighbourhood of each vertex into cliques. We will
rely on the following straightforward observation.

\begin{remark*}
  Assume we apply Lemma~\ref{lem:LLLL:new} to a negative dependency
  graph~$D=(\cX,E)$ which satisfies the following condition for some
  integers~$\ell$ and $\{q_j\}_{j\in[\ell]}$. For each vertex~$X_i$ we can
  write $\cneighbour_D(X_i)=\bigcup_{j\in[\ell]} Q_{i,j}$ where
  $|Q_{i,j}|\le q_j$ and $Q_{i,j}$ is a clique in~$D$.
  % , such that each edge in $E$ lies within some $Q_{i,j}$. 
  (Note that we do not require $Q_{i,j}\cap Q_{i,j'}=\emptyset$.) Then, if
  we replace Condition~\eqref{eq:LLLL:new:maxdeg} by
  \begin{reequation}{\ref{eq:LLLL:new:maxdeg}}
    \label{eq:LLLL:new:maxdeg'}
    \Prob(X_i)\le
    \frac{ \mu }{ \prod_{j\in[\ell]} (1+\mu q_j) } 
    \qquad\text{for all $i\in[t]$}\,,   
  \end{reequation}
  then the conclusion of Lemma~\ref{lem:LLLL:new} remains valid.

  Alternatively, assume that we are in the following
  (somewhat special) situation. There are two different types of vertices
  in~$D'$, that is, $\cX=\cX_1\dcup\cX_2$. With each type~$s\in[2]$
  of vertices we associate an integer~$\ell_s$.  Moreover, suppose that for each
  $s\in[2]$ and each $i\in[t]$ with $X_i\in\cX_s$ we can write
  $\cneighbour_D(X_i)=\bigcup_{j\in[\ell_s]} \big( Q_{i,j}^{(1)} \dcup
  Q_{i,j}^{(2)} \big)$ such that for $s'\in[2]$ we have
  $Q_{i,j}^{(s')}\subset\cX_{s'}$ and $|Q_{i,j}^{(s')}|\le q_{j,s'}$ for
  some $\{q_{j,s'}\}_{j\in[\ell_{s'}]}$.
  Assume in addition that $D\big[ Q_{i,j}^{(1)} \cup Q_{i,j}^{(2)} \big]$ is a clique.
  % such that each edge in $E$ lies within some $Q_{i,j}$, 
  Then we can replace Condition~\eqref{eq:LLLL:new:mu} by
  requiring that there are 
  positive numbers $\mu_1$ and $\mu_2$ such that for each $s\in[2]$ and
  each $X_i\in\cX_s$ we have
  \begin{reequation}{\ref{eq:LLLL:new:mu}}
  \label{eq:LLLL:new:mu'}
   \Prob(X_i)\le
   \frac{ \mu_s }{ \prod_{j\in[\ell_s]}\big(1+\mu_1 q_{j,1}+\mu_2 q_{j,2}\big) } 
   \,.
 \end{reequation}
\end{remark*}

\medskip

Next, we shall discuss how we construct negative dependency graphs in our
applications of the local lemma. For this purpose we use a framework developed
by Lu and Sz{\'e}kely~\cite{LuSze07} suited for the case when the
random experiment under consideration chooses an injection between two sets
uniformly at random.

Let~$A$ and~$B$ be two finite sets. We denote by $\Inj(A,B)$ the set of
injections from~$A$ to~$B$. In the following we consider the probability
space generated by drawing injections uniformly at random from
$\Inj(A,B)$. We shall use the following distinguished type of
events.

\begin{definition}[canonical events, conflicts]
  A \emph{canonical event~$X=X(A',B',\pi)$} for $\Inj(A,B)$ is
  determined by two sets $A'\subset A$, $B'\subset B$ and a bijection
  $\pi\colon A'\to B'$, and is defined as
  $X:=\{\sigma\in\Inj(A,B)\colon\sigma|_{A'}=\pi|_{A'}\}$.  We say that two
  canonical events $X(A'_1,B'_1,\pi_1)$ and $X(A'_2,B'_2,\pi_2)$
  \emph{conflict} if there is no injection in $\Inj(A,B)$ that extends
  both $\pi_1$ and $\pi_2$.
\end{definition}

Lu and Sz{\'e}kely showed that for canonical events there is
a simple way of constructing a negative dependency graph: we only have to
insert edges between conflicting events.

\begin{lemma}[Lu and Sz{\'e}kely~\cite{LuSze07}]
\label{lem:canon}
  Let $\cX=\{X_i\}_{i\in[t]}$ be a set of canonical events for
  $\Inj(A,B)$. Then the graph $D=(\cX,E)$ with
  \[E:=\{X_iX_j\colon\text{$X_i$ and $X_j$ conflict}\}\]
  is a negative dependency graph.
\end{lemma}

We also call the graph~$D$ constructed in this lemma \emph{canonical
  dependency graph} for~$\cX$.
For a set~$\cX$ of canonical events in $\Inj(A,B)$, let the
\emph{intersection graph} for $\cX$ be the graph $D':=(\cX,E)$ with
precisely those edges $X_1X_2$ between events $X_1=X_1(A_1,B_1,\pi_1)$ and
$X_2=X_2(A_2,B_2,\pi_2)$ with $(A_1\cap A_2)\cup(B_1\cap
B_2)\neq\emptyset$. Observe that $D'$ is a supergraph of the canonical
dependency graph for~$\cX$ and hence $D'$ is a negative dependency graph.

\section{Proofs}
\label{sec:proofs}

For the proofs of Theorems~\ref{thm:proper} and~\ref{thm:rainbow} we
combine the local lemma (Lemma~\ref{lem:LLLL:new}) with Lemma~\ref{lem:canon}.
For this purpose, given a graph~$G$ on~$n$ vertices, we let
$\cJ=\Inj\big(V(G),V(K_n)\big)$ be the set of embeddings of~$G$ into~$K_n$
(viewed as injections). Without loss of generality we assume $V(G)=V(K_n)=[n]$.

We will use the following canonical events for~$\cJ$.  For an
edge~$e=\{e_1,e_2\}$ we denote an ordered pair formed by its endvertices by
$\vec e$.  Now, let $\vec e=(e_1,e_2)$ and $\vec f=(f_1,f_2)$ be distinct
ordered pairs in~$V(G)$ such that~$e$ and~$f$ are (distinct) edges
of~$G$. In addition we require that these pairs are chosen such that
$e_1<e_2$ and $f_1<f_2$ and~$\vec e$ is lexicographically smaller
than~$\vec f$, that is, either $e_1<f_1$, or $e_1=f_1$ and $e_2<f_2$. These
inequalities merely make the choice of $e_1,e_2,f_1,f_2$ unique, given the
set of~$3$ or~$4$ vertices in the two edges of~$G$. We also say that $\vec
e$ and $\vec f$ are \emph{lexicographically ordered}.  Let~$\vec
a=(a_1,a_2)$ and~$\vec b=(b_1,b_2)$ be distinct pairs in~$V(K_n)$.
% Somewhat imprecisely, we also call~$e$ and~$f$ edges of~$G$, and~$a$
% and~$b$ edges of~$K_n$, and consider them as sets when this is convenient.

We denote by $X(\vec e, \vec f; \vec a, \vec b)$ the event containing all
embeddings of~$G$ into~$K_n$ which map~$\vec e$ to~$\vec a$, and~$\vec f$
to~$\vec b$, that is, $X(\vec e, \vec f; \vec a, \vec b)$ is the canonical
event $X(e\cup f, a\cup b,\pi)$ with $\pi(e_1)=a_1$, $\pi(e_2)=a_2$,
$\pi(f_1)=b_1$, and $\pi(f_2)=b_2$.  It goes without saying that we only
consider events $X(\vec e, \vec f; \vec a, \vec b)$ for which such an
injection $\pi$ exists.  We say that $X(\vec e, \vec f; \vec a, \vec b)$ is
of \emph{disjoint type} if~$e$ and~$f$ (and hence also~$a$ and~$b$) are
disjoint, otherwise we say that $X(\vec e, \vec f; \vec a, \vec b)$ is of
\emph{intersecting type}. We have
\begin{equation}
\label{eq:canon:P}
\Prob\big(X(\vec e, \vec f; \vec a, \vec b)\big)=\begin{cases}
  1\,/\,(n)_3 &
  \quad\text{if $X(\vec e, \vec f; \vec a, \vec b)$ is of intersecting type,} \\
  1\,/\,(n)_4 &
 \quad\text{otherwise,}
\end{cases}
\end{equation}
where $(n)_k=n!/(n-k)!$ is the falling factorial.

Observe that two canonical events $X(\vec e, \vec f; \vec a, \vec b)$ and
$X(\vec e\,', \vec f'; \vec a', \vec b')$ with bijections $\pi$ and $\pi'$
conflict only if one of the following cases occurs. Either there is a
vertex in $(e\cup f)\cap(e'\cup f')$ for which $\pi$ and $\pi'$ are
inconsistent, or there is a vertex in $(a\cup b)\cap(a'\cup b')$ for which
$\pi$ and $\pi'$ are inconsistent. In the first case $e\cup f$ and $e'\cup
f'$ share some element and in the second case $a\cup b$ and $a'\cup b'$
share some element. This motivates the following definitions. We say that
the events $X(\vec e, \vec f; \vec a, \vec b)$ and $X(\vec e\,', \vec f';
\vec a', \vec b')$ are \emph{$G$-intersecting} if $(e\cup f)\cap(e'\cup
f')\neq\emptyset$ and that they are \emph{$K_n$-intersecting} if $(a\cup
b)\cap(a'\cup b')\neq\emptyset$.

\subsection{Properly coloured subgraphs}

In this section we prove Theorem~\ref{thm:proper}. For this purpose we
take a random embedding of~$G$ into~$K_n$. A `bad event' occurs if two
adjacent edges~$e$ and~$f$ of~$G$ are mapped onto two (adjacent) edges
of~$K_n$ with the same colour. We will use the local lemma to show that
the probability that none of those bad events occurs is positive.

\begin{proof}[Proof of Theorem~\ref{thm:proper}]
  % Observe that for $n<10$ the condition of Theorem~\ref{thm:rainbow} cannot
  % be satisfied. Hence we assume $n\ge 10$.
  Given~$G$ and a locally $k$-bounded edge colouring~$\chi$ of~$K_n$, we
  consider a random embedding~$\sigma\colon V(G)\to V(K_n)$ of~$G$
  into~$K_n$ (that is, a random injection from the set~$\cJ$ defined above)
  and show that with positive probability~$\sigma$ has the
  desired property. 

  Let the canonical events $X(\vec e, \vec f; \vec a, \vec b)$ with $e,f\in
  E(G)$ and $a,b\in E(K_n)$ of~$\cJ$ be as defined above. Let the set of
  bad events~$\cX$ be the set of events~$X(\vec e, \vec f; \vec a, \vec b)$
  of intersecting type such that $\chi(a)=\chi(b)$ and let~$D'$ be the
  intersection graph for~$\cX$.  Observe that, if no bad event occurs,
  then~$\sigma$ provides a properly coloured copy of~$G$.  Our goal is to
  apply version~\ref{lem:LLLL:new:maxdeg} of Lemma~\ref{lem:LLLL:new} to
  show that $\Prob(\bigcap_{X\in\cX}\BAR{X})>0$.
  For this purpose it suffices to check
  Condition~\eqref{eq:LLLL:new:maxdeg'}. Therefore we will next analyse the
  closed neighbourhood $\cneighbour_{D'}(X)$ of events~$X$ in~$D'$.

  Let $X=X(\vec e, \vec f; \vec a, \vec b) \in\cX$ be fixed. 
  Let~$S_G(X)$ be the set of events in~$\cX$ that are
  $G$-intersecting with~$X$ and~$S_{K_n}(X)$ be the set of events
  in~$\cX$ that are $K_n$-intersecting with~$X$. Then
  $\cneighbour_D(X)=\{X\}\dcup S_{G}(X)\dcup S_{K_n}(X)$.

  Observe that $e$ and $f$ form a cherry in~$G$. Call its three vertices
  $x,y,z$.  Assume that $X(\vec e\,', \vec f'; \vec a', \vec b')\in\cX$ is
  $G$-intersecting with $X$. Then the cherry $e',f'$ contains some vertex
  $x'\in\{x,y,z\}$. Recall that~$x'$ is contained in at most~$q$ cherries
  of~$G$. Hence, given~$x'$ there are at most~$q$ choices for $e',f'$. Once
  $\{e',f'\}$ has been fixed, the vertices of $\vec a'$ can be chosen in
  $(n)_2$ ways, and for fixed $\vec a'$ there are at most~$k$ choices for
  the third vertex in~$a'\cup b'$ since $\chi(a')=\chi(b')$ and $\chi$ is
  $k$-bounded.  It follows that $\{X\}\dcup S_G$ is the union of three
  (overlapping) cliques in~$D'$ of order at most $q\cdot (n)_2k$ each.

  Similarly, $a$ and $b$ form a cherry. Call its vertices $u,v,w$. If
  $X(\vec e\,', \vec f'; \vec a', \vec b')\in\cX$ is $K_n$-intersecting
  with $X(\vec e, \vec f; \vec a, \vec b)$, then the cherry $a',b'$
  contains some vertex $u'\in\{u,v,w\}$.  The number of monochromatic
  cherries in $K_n$ containing~$u'$ as an end point is at most $(n-1)k$ and
  of those containing~$u'$ as the middle point is at most
  $\frac12(n-1)k$. Moreover, there are~$2$ injections from the vertices of
  a cherry in~$G$ to such a monochromatic cherry and there are at most $pn$
  cherries in~$G$ by hypothesis. We infer that $S_{K_n}(X)$ can be written
  as the union of three (overlapping) cliques in~$D'$ of order at most
  $2\big((n-1)k+\frac12(n-1)k\big)\cdot pn=3p\cdot (n)_2k$ each.

  In conclusion $\cneighbour_{D'}\big(X(\vec e, \vec f; \vec a, \vec
  b)\big)$ is the union of three
  cliques of order at most $q(n)_2k$ and three cliques of order at most
  $3p(n)_2k$. For checking
  Condition~\eqref{eq:LLLL:new:maxdeg'}, it is enough to show that there
  is a positive number~$\mu$ such that
  \begin{equation*}
  \Prob\big(X(\vec e, \vec f; \vec a, \vec b)\big)
    \eqByRef{eq:canon:P} \frac{1}{(n)_3}
    \le \frac{\mu}{\big(1+q\cdot (n)_2k\mu\big)^3\big(1+3p\cdot (n)_2k\mu\big)^3} \;.
  \end{equation*}
  We claim that for $\mu:=\big(\frac65\big)^6/(n)_3$ this inequality is satisfied. Indeed,
  since $k\le \frac13\big(\frac56\big)^5(n-2)/(q+3p)$ we have
  $k\mu\le\frac25/\big((n)_2(q+3p)\big)$. This implies
  \begin{multline*}
    \big(1+q\cdot (n)_2k\mu\big)^3\big(1+3p\cdot (n)_2k\mu\big)^3
    \le \Bigg(
        \bigg( 1+\frac25\cdot\frac{q}{q+3p} \bigg)
        \bigg( 1+\frac25\cdot\frac{3p}{q+3p} \bigg)
      \Bigg)^3 \\
    = \bigg( 1+\frac25+\frac4{25}\cdot\frac{3qp}{(q+3p)^2}\bigg)^3
    \le \bigg( 1+\frac25+\frac4{25}\cdot\frac14\bigg)^3
    =\bigg(\frac65\bigg)^6 \,,
  \end{multline*}
  where in the last inequality we used that $st/(s+t)^2\le 1/4$
  for all reals~$s$ and~$t$ with $s+t\neq 0$.
  Therefore we obtain
  \begin{equation*}
    \frac{\mu}{\big(1+q(n)_2k\mu\big)^3\big(1+3p(n)_2k\mu\big)^3} 
    \ge \frac{\big(\frac65\big)^6}{\big(\frac65\big)^6\cdot(n)_3}
    = \Prob\big(X(\vec e, \vec f; \vec a, \vec b)\big)\;,
  \end{equation*}
  as required.
\end{proof}

  % We claim that for $\mu:=3/(n)_3$ this inequality is satisfied. Indeed,
  % since $k\le (n-2)/8.2(q+3p)$ we have
  % \begin{equation*}
  %   \big(1+q(n)_2k\mu\big)^3\big(1+3p(n)_2k\mu\big)^3
  %   \le \exp\big(3(q+3p)(n)_2k\mu\big) 
  %   \le \exp(3\cdot 3/ 8.2)
  %   < 3
  % \end{equation*}
  % and therefore
  % \begin{equation*}
  %   \frac{\mu}{\big(1+q(n)_2k\mu\big)^3\big(1+3p(n)_2k\mu\big)^3} 
  %   > \frac{3}{3(n)_3}
  %   = \Prob\big(X(\vec e, \vec f; \vec a, \vec b)\big)\;,
  % \end{equation*}
  % as required.
 
  % {\red Put now $k=c(n-2)$ and $\mu:=\frac{\nu}{c(q+3p)(n)_3}$. Then inequality  above is satisfied for $c$ such that
  %   \begin{equation*}
  %     c\;\;
  %     {\red \le} \; {1\over q+3p}\;\frac{\nu}{\big(1+\nu+ {3pq\over (q+3p)^2}\nu^2\big)^3} \;
  %   \end{equation*}
  %   Observing now that for any positive $x,y$ we have that $xy/(x+y)^2\le 1/4$, inequality above is satisfied if
  %   \begin{equation*}
  %     c (q+3p)\;\;
  %     {\red \le} \; \;\frac{\nu}{\big(1+\nu+ {\nu^2\over 4}\big)^3} \;\le \;\frac{2/5}{\big(1+{2\over 5}+ {1\over 25}\big)^3} = {1\over 7,46496}
  %   \end{equation*}
  %   Therefore, for $k={n-2\over 7,46496(q+3p)}$the   Condition~\eqref{eq:LLLL:new:maxdeg'} is satisfied.
  % }

\subsection{Rainbow subgraphs}

In this section we prove Theorem~\ref{thm:rainbow}. Its proof follows the
strategy of the proof for Theorem~\ref{thm:proper}. The major difference is
that now we have to consider both, canonical events of intersecting and of
disjoint type. This is why we will consider bad events with very different
probabilities and therefore we will apply version~\ref{lem:LLLL:new:mu} of the
local lemma in the form of Lemma~\ref{lem:LLLL:new}.

\begin{proof}[Proof of Theorem~\ref{thm:rainbow}]
  Observe that for $n<77$ the condition of Theorem~\ref{thm:rainbow}
  implies~$k\le1$ and therefore in this case the theorem holds
  trivially. Hence we assume $n\ge 77$.

  Given~$G$ and a $k$-bounded edge colouring~$\chi$ of~$K_n$,
  let~$\cJ$ and its canonical events $X(\vec e, \vec f; \vec a, \vec b)$ be as defined at the
  beginning of Section~\ref{sec:proofs}, and let
  $\sigma\colon V(G)\to V(K_n)$ be a random injection from~$\cJ$.
  Let~$m$ denote the number of edges in~$G$.
  
  Let the set of bad events~$\cX$ be the set of canonical
  events $X(\vec e, \vec f; \vec a, \vec b)$ (of intersecting and disjoint type) such that
  $\chi(a)=\chi(b)$ and let~$D'$ be the intersection graph
  for~$\cX$. Again, if no bad event occurs, then~$\sigma$ gives a
  rainbow copy of~$G$. Hence, if we can apply version~\ref{lem:LLLL:new:mu} of
  Lemma~\ref{lem:LLLL:new} to show that $\Prob(\bigcap_{B\in\cX}\BAR{B})>0$, we
  are done.

  For Lemma~\ref{lem:LLLL:new}\ref{lem:LLLL:new:mu} it suffices to define for every
  bad event $B\in\cX$ a positive number~$\mu_B$ such that~\eqref{eq:LLLL:new:mu'}
  is satisfied . In order to find out how we should set these numbers we
  will investigate the neighbourhood of~$B$ in~$D'$, write it as a union of
  cliques, and derive bounds on their sizes. 

  Let $X=X(\vec e, \vec f; \vec a, \vec b)$ be an arbitrary bad event.
  If~$X$ is of intersecting type then
  we have $e\cup f=\{x_1,x_2,x_3\}$ and $a\cup b=\{u_1,u_2,u_3\}$;
  otherwise we have $e\cup f=\{x_1,x_2,x_3,x_4\}$ and $a\cup
  b=\{u_1,u_2,u_3,u_4\}$.
  In either case, for~$\ell\in[4]$,
  let $Q_G^{(\ell)}$ denote the set of neighbours of~$X$ in~$D'$ which
  are $G$-intersecting with~$X$ and contain~$x_\ell$. Analogously, 
  let~$Q_{K_n}^{(\ell)}$ denote the set of neighbours of~$X$ in~$D'$ which
  are $K_n$-intersecting and contain~$u_\ell$.
  Hence we have
  \begin{equation*}
    \neighbour_{D'}(X)=
    \begin{cases}
      \bigcup\nolimits_{\ell\in[3]} 
      \Big(Q_G^{(\ell)}\cup Q_{K_n}^{(\ell)} \Big) &
    \quad\text{if $X$ is of intersecting type}\,, \\[2mm]
      \bigcup\nolimits_{\ell\in[4]} 
      \Big(Q_G^{(\ell)}\cup Q_{K_n}^{(\ell)} \Big) &
     \quad\text{if $X$ is of disjoint type} \,.
    \end{cases}
  \end{equation*}
  Moreover, we write $Q_{\its(G)}^{(\ell)}$ for the set of those events
  in~$Q_G^{(\ell)}$ that are of intersecting type and
  $Q_{\dis(G)}^{(\ell)}$ for the set of those of disjoint type;
  therefore $Q_{G}^{(\ell)}=Q_{\its(G)}^{(\ell)} \dcup
  Q_{\dis(G)}^{(\ell)}$.  Analogously
  $Q_{K_n}^{(\ell)}=Q_{\its(K_n)}^{(\ell)} \dcup Q_{\dis(K_n)}^{(\ell)}$.

  \begin{AuxClaim}\label{cl:rainbow:Q}
    For each~$\ell$ the sets $\{X\}\cup Q_G^{(\ell)}$ and
    $Q_{K_n}^{(\ell)}$ are cliques in~$D'$ and we have
    \begin{equation}\label{eq:rainbow:Q}
      \begin{aligned}
        \big| \{X\}\cup Q_{\its(G)}^{(\ell)} \big| &\le \gamma_{\its} :=\tfrac32\Delta^2n^2 k \,, &\qquad
        \big| \{X\}\cup Q_{\dis(G)}^{(\ell)} \big| &\le \gamma_{\dis} :=\Delta^2 n^3 k \,, \\
        \big| Q_{\its(K_n)}^{(\ell)} \big| &\le \kappa_{\its} := \Delta^2 n^2 k\,, &\qquad
        \big| Q_{\dis(K_n)}^{(\ell)} \big| &\le \kappa_{\dis} :=\Delta^2 n^3 k \,.
      \end{aligned}
    \end{equation}
  \end{AuxClaim}
  \begin{proof}
    Clearly $\{X\}\cup Q_G^{(\ell)}$ is a clique in~$D'$ since each of its events
    contains~$x_\ell$, and $Q_{K_n}^{(\ell)}$ is a clique since each of its
    events contains~$u_\ell$.

    The bound on the number of events $X'=X(\vec e\,', \vec f'; \vec a',
    \vec b')$ in $\{X\}\cup Q_{\its(G)}^{(\ell)}$ follows from the following facts.
    We have $x_\ell\in e'\cup f'$ and there are at most $\frac32\Delta^2$
    cherries in~$G$ containing~$x_{\ell}$ (recall that $\vec e\,'$ and $\vec
    f'$ are lexicographically ordered). The two edges forming the cherry
    $a',b'$ can be chosen in $\binom{n}{2}k$ ways, and there are~$2$
    isomorphisms from the cherry $e',f'$ to this cherry.

    In $\{X\}\cup Q_{\dis(G)}^{(\ell)}$ there are at most $\Delta m \cdot (n)_2(2k)$
    events $X'=X(\vec e\,', \vec f'; \vec a', \vec b')$ since there are at
    most $\Delta$ ways to form an edge in~$G$ containing~$x_\ell$. For the
    second edge of $e',f'$ there are at most~$m$ possibilities. Moreover,
    for~$\vec a'$ we have $(n)_2$ choices.  Since $\chi(a')=\chi(b')$ this
    leaves at most~$k$ choices for $b'$ and there are~$2$ ways to
    obtain~$\vec b'$ from~$b'$. The bound claimed in~\eqref{eq:rainbow:Q}
    then follows from $m\le\frac12\Delta n$.

    For events $X'=X(\vec e\,', \vec f'; \vec a', \vec b')$ in
    $Q_{\its(K_n)}^{(\ell)}$ observe that there are at most
    $\binom{\Delta}{2}n\le\frac12\Delta^2n$ cherries in~$G$. It follows that we can
    choose $\vec e\,' , \vec f'$ in at most $\frac12\Delta^2 n$ ways.
    Moreover, $u_\ell\in a'\cup b'$. For finding a monochromatic cherry
    in~$K_n$ which contains~$u_\ell$ we have at most~$nk$ choices, and
    there are~$2$ isomorphisms from the cherry $e',f'$ to this
    cherry. Hence there are at most $\Delta^2 n^2 k$ such events.

    Finally, the number of events $X'=X(\vec e\,', \vec f'; \vec a', \vec
    b')$ contained in $Q_{\dis(K_n)}^{(\ell)}$ is at most $\frac18\Delta^2 n^2\cdot
    4\cdot n(2k)$ because there are at most $\binom{m}{2}\le\frac18\Delta^2
    n^2$ possibilities to choose $\vec e\,',\vec f'$, and because $u_\ell$
    can be each of the~$4$ vertices in $a'\cup b'$. For the second
    vertex in the edge of~$X'$ which contains~$x$, say~$a'$, there are then
    at most~$n$ possibilities, and for $\vec b'$ we have again~$2k$ choices.
  \end{proof}

  Now we define for each $B\in\cX$ the number
  \begin{equation}
    \label{eq:rainbow:mu}
    \mu_B:=\begin{cases}
      \mu_{\its} := \big( \frac{14}{10n} \big)^3 & \quad\text{if $B$ is of intersecting type} \\
      \mu_{\dis} := \big( \frac{14}{10n} \big)^4 & \quad\text{if $B$ is of disjoint type} \\
    \end{cases}
  \end{equation}
  and claim that with these numbers Condition~\eqref{eq:LLLL:new:mu'} is
  satisfied. Indeed, let $B_{\its}\in\cX$ be an arbitrary intersecting event
  and $B_{\dis}\in\cX$ an arbitrary disjoint event. Claim~\ref{cl:rainbow:Q}
  and Condition~\eqref{eq:LLLL:new:mu'} imply together with~\eqref{eq:rainbow:mu}
  that it is sufficient to check the two conditions
  \begin{reequation}{\ref{eq:LLLL:new:mu'}}
    \label{eq:rainbow:LLLL}
    \begin{aligned}
      \Prob(B_{\its})
      &\le \frac{ \mu_{\its} }
      { ( 1+ \gamma_{\its}\mu_{\its}+\gamma_{\dis}\mu_{\dis} )^3
        ( 1+ \kappa_{\its}\mu_{\its} +\kappa_{\dis}\mu_{\dis})^3
      } =: p_{\its}  \,, \\
      \Prob(B_{\dis})
      &\le \frac{ \mu_{\dis} }
      { ( 1 + \gamma_{\its}\mu_{\its}+\gamma_{\dis}\mu_{\dis} )^4
        ( 1 + \kappa_{\its}\mu_{\its}+\kappa_{\dis}\mu_{\dis} )^4
      } =: p_{\dis} \,.
      % \Prob(B_{\its})
      % &\le \frac{ \mu_{\its} }
      % { ( 1+\Delta^2n^3 k \mu_{\dis}+ 2\Delta^2n^2 k\mu_{\its} )^3
      % ( 1+\Delta^2n^3k\mu_{\dis}+ \frac32\Delta^2n^2k\mu_{\its})^3
      % }  \,, \\
      %   \Prob(B_{\dis})
      %   &\le \frac{ \mu_{\dis} }
      %   { ( 1+\Delta^2n^3 k \mu_{\dis}+ 2\Delta^2n^2 k\mu_{\its} )^4 
      %   ( 1+\Delta^2n^3k\mu_{\dis}+ \frac32\Delta^2n^2k\mu_{\its})^4
      % }  \,, \\
    \end{aligned}
  \end{reequation}
  Since $k\le n/(51\Delta^2)$
  we have
  \begin{multline*}
    ( 1 + \gamma_{\its}\mu_{\its}+\gamma_{\dis}\mu_{\dis} )
    ( 1 + \kappa_{\its}\mu_{\its}+\kappa_{\dis}\mu_{\dis} ) \\
    \begin{aligned}
      & \leByRef{eq:rainbow:Q}
      ( 1+ \tfrac32\Delta^2n^2k\mu_{\its}+\Delta^2n^3k\mu_{\dis})
      ( 1+ \Delta^2n^2 k\mu_{\its}+\Delta^2n^3 k \mu_{\dis} ) \\
     & \leByRef{eq:rainbow:mu}
      \Big( 1 + 
        \tfrac32\cdot\tfrac1{51}\big(\tfrac{14}{10}\big)^3 +
        \tfrac1{51}\big(\tfrac{14}{10}\big)^4
      \Big)
      \Big( 1 + 
        \tfrac1{51}\big(\tfrac{14}{10}\big)^3 +
        \tfrac1{51}\big(\tfrac{14}{10}\big)^4
      \Big)
      \le \tfrac{50}{51}\cdot\tfrac{14}{10} \,,
    \end{aligned}
  \end{multline*}
  where the last inequality can be easily verified numerically.
  This implies the second part of~\eqref{eq:rainbow:LLLL} because
  \begin{align*}
    p_{\dis} 
    \ge \frac{\mu_{\dis}}{ \big(\tfrac{50}{51}\cdot\tfrac{14}{10} \big)^4}
    \eqByRef{eq:rainbow:mu} \Big(\frac{51}{50n}\Big)^4
    \ge \frac{1}{(n)_4}
    \eqByRef{eq:canon:P}\Prob(B_{\dis}) \,,
  \end{align*}
  where we used $n\ge 77$ in the last inequality.
 The first part of~\eqref{eq:rainbow:LLLL} follows analogously.
  Hence we can apply Lemma~\ref{lem:LLLL:new}\ref{lem:LLLL:new:mu} to conclude
  that $\Prob(\bigcap_{B\in\cX}\BAR{B})>0$, which finishes our proof.
\end{proof}

\section{Concluding remarks}

In this paper we showed how the framework developed by Lu and
Sz\'ekely~\cite{LuSze07} for applying the local lemma to random injections 
can be used for obtaining results about copies of spanning graphs in
bounded or locally bounded edge colourings of~$K_n$. 
In our proofs we used the version of the local lemma given in
Lemma~\ref{lem:LLLL:new}, which enabled us to obtain better constants than
Lemma~\ref{lem:LLLL} would have yielded.
We close with two remarks:
\begin{enumerate}[label=\rom,leftmargin=*]
  \item If $n\ge 100$ then in Theorem~\ref{thm:rainbow} the constant~$51$ can be improved
    to~$42$, using calculations analogous to those in the proof of
    Theorem~\ref{thm:rainbow}.
  \item Albert, Frieze, and Reed~\cite{AlbFriRee95} used the local lemma in the
    form of Lemma~\ref{lem:LLLL} to obtain
    Theorem~\ref{thm:Albert}. Using Lemma~\ref{lem:LLLL:new} instead, one
    can improve the constant~$64$ in Theorem~\ref{thm:Albert} to~$38$,
    if~$n$ is sufficiently large.
\end{enumerate}

\section{Acknowledgements}

We thank Noga Alon for helpful discussions.

%%%% BIBLIOGRAPHY %%%%%%%%%%%%%%%%%%%%%%%%%%%%%%%%%%%%%%%%%%%%%%%%%%%%%

\bibliographystyle{amsplain_yk}
\bibliography{rainbow}

\end{document}